\documentclass[]{amsart}
\usepackage{amscd,amsthm,amssymb,amsfonts,amsmath,euscript}

\theoremstyle{plain}
\newtheorem{thm}{Theorem}[section]
\newtheorem{lemma}[thm]{Lemma}

\newtheorem{cor}[thm]{Corollary}

\theoremstyle{definition}

\theoremstyle{remark}

\newcommand{\nc}{\newcommand}

\def\makeop#1{\expandafter\def\csname#1\endcsname
  {\mathop{\rm #1}\nolimits}\ignorespaces}
\makeop{Hom}   \makeop{End}   \makeop{Aut}   \makeop{Isom}  \makeop{Pic} 
\makeop{Gal}   \makeop{ord}   \makeop{Char}  \makeop{Div}   \makeop{Lie} 
\makeop{PGL}   \makeop{Corr}  \makeop{PSL}   \makeop{sgn}   \makeop{Spf}
\makeop{Spec}  \makeop{Tr}    \makeop{Nr}    \makeop{Fr}    \makeop{disc}
\makeop{Proj}  \makeop{supp}  \makeop{ker}   \makeop{im}    \makeop{dom}
\makeop{coker} \makeop{Stab}  \makeop{SO}    \makeop{SL}    \makeop{SL}
\makeop{Cl}    \makeop{cond}  \makeop{Br}    \makeop{inv}   \makeop{rank}
\makeop{id}    \makeop{Fil}   \makeop{Frac}  \makeop{GL}    \makeop{SU}
\makeop{Nrd}   \makeop{Sp}    \makeop{Tr}    \makeop{Trd}   \makeop{diag}
\makeop{Res}   \makeop{ind}   \makeop{depth} \makeop{Tr}    \makeop{st}
\makeop{Ad}    \makeop{Int}   \makeop{tr}    \makeop{Sym}   \makeop{can}
\makeop{length}\makeop{SO}    \makeop{torsion} \makeop{GSp} \makeop{Ker}
\makeop{Adm}   \makeop{Mat}
\def\makebb#1{\expandafter\def
  \csname bb#1\endcsname{{\mathbb{#1}}}\ignorespaces}
\def\makebf#1{\expandafter\def\csname bf#1\endcsname{{\bf
      #1}}\ignorespaces} 
\def\makegr#1{\expandafter\def
  \csname gr#1\endcsname{{\mathfrak{#1}}}\ignorespaces}
\def\makescr#1{\expandafter\def
  \csname scr#1\endcsname{{\EuScript{#1}}}\ignorespaces}
\def\makecal#1{\expandafter\def\csname cal#1\endcsname{{\mathcal
      #1}}\ignorespaces} 
\def\doLetters#1{#1A #1B #1C #1D #1E #1F #1G #1H #1I #1J #1K #1L #1M
                 #1N #1O #1P #1Q #1R #1S #1T #1U #1V #1W #1X #1Y #1Z}
\def\doletters#1{#1a #1b #1c #1d #1e #1f #1g #1h #1i #1j #1k #1l #1m
                 #1n #1o #1p #1q #1r #1s #1t #1u #1v #1w #1x #1y #1z}
\doLetters\makebb   \doLetters\makecal  \doLetters\makebf
\doLetters\makescr 
\doletters\makebf   \doLetters\makegr   \doletters\makegr
     \def\qed{\qedmark\medbreak}%
\def\qedmark{{\enspace\vrule height 6pt width 5pt depth 1.5pt}}%
    
\def\Gm{{{\bbG}_{\rm m}}}   

\normalsize

\makeop{Bl}

\def\Qbar{\overline{\bbQ}}

\newcommand{\Z}{\mathbb Z}
\newcommand{\Q}{\mathbb Q}
\newcommand{\R}{\mathbb R}
\newcommand{\C}{\mathbb C}
\newcommand{\pr}{\indent }

\newcommand{\<}{\langle}   %\< is not defined yet.
\renewcommand{\>}{\rangle} %\> is already defined.

\nc{\embed}{\hookrightarrow}
\newcommand{\ch}{characteristic }

\nc{\ol}{\overline}
\nc{\wt}{\widetilde}
\nc{\opp}{\mathrm{opp}}

\makeop{Ram}
\makeop{Rep}

\begin{document}
\renewcommand{\thefootnote}{\fnsymbol{footnote}}
\setcounter{footnote}{-1}
\numberwithin{equation}{section}
%\numberwithin{section}{chapter}

%\usepackage[notref,notcite]{showkeys}

\title[Mumford-Tate conjecture]
{A note on the Mumford-Tate Conjecture for CM abelian varieties}
\author{Chia-Fu Yu}
\address{
Institute of Mathematics, Academia Sinica and NCTS (Taipei Office)\\
6th Floor, Astronomy Mathematics Building \\
No. 1, Roosevelt Rd. Sec. 4 \\ 
Taipei, Taiwan, 10617} 
\email{chiafu@math.sinica.edu.tw}
\address{
The Max-Planck-Institut f\"ur Mathematik \\
Vivatsgasse 7, Bonn \\
Germany 53111} 
\email{chiafu@mpim-bonn.mpg.de}
%\date{June 23, 2000}

\date{\today}

\begin{abstract}
The Mumford-Tate conjecture is first proved for CM abelian varieties by
H. Pohlmann [Ann. Math., 1968]. 
In this note we give another proof of this result and extend it 
to CM motives. 
%The main ingredient of the proof
%is the Taniyam-Shimura theory of complex multiplication of abelian
%varieties, which is the same as used in Pohlmann's proof.   
% The proof uses the same ideas of Pohlmann (the
% Taniyama-Shu
% This is known to experts but not yet well documented in the literature.
\end{abstract} 

\maketitle

%\tableofcontents   % Table of Contents

\section{Introduction }
\label{sec:01}

Let $A$ be an abelian variety over a field $k$, where $k$ is finitely
generated over $\Q$ and is contained in the field $\C$ of complex
numbers. 
% that is finitely generated over $\Q$. 
% is contained in the field $\C$ of complex numbers. 
The
Mumford-Tate group $MT(A)$ of $A_\C:=A\otimes \C$ 
is defined to be the smallest algebraic $\Q$-subgroup $G$ of $\GL(V)$ 
such that $G(\C)$ contains the image of the 
Hodge cocharacter $\mu:\C^\times \to
\GL(V_\C)$, where $V:=H_1(A(\C), \Q)$ is the first rational homology
group of $A$ and $V_\C=V\otimes \C$. 
For any rational prime $\ell$, let
$G_{A,\ell}$ denote the algebraic envelope of the associated 
$\ell$-adic Galois
representation 
\[ \rho_\ell:\Gal(\bar k/k)\to \GL(V_\ell(A)), \]
where $V_\ell(A):=T_\ell(A)\otimes_{\Z_\ell} \Q_\ell$,  $T_\ell(A)$
is the $\ell$-adic Tate module of $A$ and $\bar k$ is the
algebraic closure of $k$ in $\C$. That is, $G_{A,\ell}$ is the
Zariski closure of the image 
$\rho_\ell(\Gal(\bar k/k))$ in the algebraic $\Q_\ell$-group 
$\GL(V_\ell(A))$. Let $G_{A,\ell}^0$ denote the neutral component of 
the algebraic group $G_{A,\ell}$ over $\Q_\ell$.

Under the comparison isomorphism 
$V_\ell:=V\otimes \Q_\ell \simeq V_\ell(A)$ and hence $\GL(V)\otimes
\Q_\ell$ being identified with the $\Q_\ell$-group $\GL(V_\ell(A))$, 
the Mumford-Tate conjecture (MTC) asserts the equality of these two
algebraic subgroups
\begin{equation}
  \label{eq:1.1}
  MT(A)\otimes \Q_\ell=G_{A,\ell}^0
\end{equation}
of $\GL(V_\ell)$ for all primes $\ell$. 

Observe that the group $G_{A,\ell}^0$ is unchanged if one replaces 
the field $k$ by a finitely generated field extension $k_1$ of $k$ in
$\C$. Indeed, as one has the natural identification 
$T_\ell(A\otimes \bar k)=T_\ell(A\otimes \bar k_1)$, the action of
the Galois group $\Gal(\bar k_1/k_1)$ on the Tate module 
factors through its quotient $\Gal(\bar k/\bar k\cap k_1)$. Therefore, 
the group $G_{A,\ell}^0$ is unchanged if one
replaces $A$ by $A\otimes_k k_1$ since it remains the same after a
finite field extension base change. 
In particular, the MTC does not depend on
the choice of the finitely generated field $k$.

The meaning of the MTC asserts that  
the algebra of all $\ell$-adic Tate cohomology classes (with respect 
to all finite algebraic extensions of $k$) 
on every self-product $A^m$ of $A$
coincides with the $\Q_{\ell}$-algebra generated by Hodge classes of
$A_\C$. For CM abelian varieties, this statement is proved 
by H. Pohlmann \cite[Theorem 5]{pohlmann} and hence the MTC follows.

In this expository article we give another proof of 
the MTC for abelian varieties of CM type. 
Observe that both groups concerned are subtori of the algebraic
torus associated to the CM field in question 
(also see Section~\ref{sec:2.2}). 
We prove the equality by checking their
co-character groups. 
The proof is based on
the Taniyama-Shimura theory of complex multiplication of abelian
varieties (\cite{shimura-taniyama} also see \cite{serre-tate} and 
\cite{shimura:cm1998}), 
which is the same as used in Pohlmann's proof. Therefore, we do not
claim any novelty in this expository article. Note that although
\cite{pohlmann} is the main reference where the first case of MTC is
proved, the word ``Mumford-Tate'' does not appear in the paper. We
also found the review article by Ribet \cite[p.~216]{ribet:survey}
where he pointed out that this is a corollary of the main theorem of 
complex multiplication due to Shimura-Taniyama~\cite{shimura-taniyama}
(but without referring to Pohlmann's work). This is indeed the case, and
we are simply adding more details to this ``corollary'' here 
for the reader's convenience.

The Mumford-Tate conjecture has been verified for a large class of
important cases. 
It is proved by Deligne~\cite[I. Proposition 6.2]{dmos} 
 that the inclusion 
\[ G^0_{A,\ell}\subset MT(A)\otimes \Q_\ell\]
holds for all primes $\ell$. 
Serre has established several fundamental
tools and methods for analyzing $\ell$-adic Galois representations 
(see \cite{serre:motivic1994} and references given there). 
Among other results 
he~\cite{serre:1984-5} proved the MTC for abelian
varieties $A$ with $\End_{\bar k}(A)=\Z$ and $\dim A=2$, $6$ or odd.
Ribet has developed several methods and among other results
he proved the MTC for abelian varieties with 
real multiplication (the generic case); 
see~\cite{ribet:ajm1976, ribet:compos1976,ribet:smf1980}. 
The centralizers of $G_{A,\ell}$ and of
$MT(A)\otimes \Q_\ell$ coincide in $\GL(V_\ell)$; this follows from
Faltings' theorem~\cite{faltings:end} for Tate's conjecture on
homomorphisms of abelian varieties. 
The rank of $G_{A,\ell}$ (the dimension of a maximal subtorus) 
is independent of primes $\ell$ (see Serre~\cite[2.2.4,
p.~31]{serre:1984-5}).   
If one has the equality 
$G^0_{A,\ell}= MT(A)\otimes \Q_\ell$ for one prime $\ell$, then one has
the quality $G^0_{A,\ell}= MT(A)\otimes \Q_\ell$ for all primes $\ell$
(due to Tankeev~\cite{tankeev:96, tankeev:izv1996}, Larsen and
Pink~\cite{pink:l_mono1998,larsen-pink:mathann1998}). 

% For some recent developments of the Mumford-Tate conjecture, we refer 
% to the works of Banaszak, Gajda and Kraso\'n~\cite{bgk:rm,
%   bgk:I_II, bgk:III} and of Vasiu~\cite{vasiu:mt}.

We remark that the Mumford-Tate group $MT(A)$ is the fundamental group 
of the Tannakian category generated by the Hodge structure 
$V=H_1(A_\C,\Q)$ and the Tate twists $\Q(1)$ (see \cite{dmos}). 
Similarly the group $G_{A,\ell}^0$ is the fundamental group of 
the Tannakian category generated by the 
$\rho_\ell(\Gal(\bar k/k_1))$-representation $V_\ell(A)$ and the
Tate twist $\Q_\ell(1)$, where $k_1$ is any finite extension of $k$ so
that $\rho_\ell (\Gal(\bar k/k_1))$ is contained in $G_{A,\ell}^0(\Q_\ell)$. 
Thus, the equality (\ref{eq:1.1}) asserts the equivalence of these
two Tannakian categories. Clearly these groups can be 
defined also for more general projective smooth varieties
through the singular cohomology and etale cohomology, and hence the MTC
can be formulated for more general varieties even for motives. 
In the last section, we 
show how the MTC for CM motives follows from that for CM abelian
varieties.

\section{Algebraic Tori}
\label{sec:02}

\subsection{Basic properties of $k$-tori}
\label{sec:2.1}
Let $k$ be any field of \ch zero. Let $\Gamma_k:=\Gal(\bar k/k)$
denote the absolute Galois group of $k$, where $\bar k$ is an
algebraic closure 
of $k$. Let $(\text{Diag.~groups}/k)$ denote the category of
diagonalizable groups over $k$ (see \cite{borel:lag}); this is an abelian
category. Let (Tori/$k$) denote the category of algebraic tori over
$k$; this is a full subcategory of $(\text{Diag. groups}/k)$ but not
an abelian category. 
Let $(\Z[\Gamma_k]\text{-mod})$ denote the abelian category
of finitely generated $\Z$-modules $X$ together with a continuous
action of $\Gamma_k$ (with the discrete topology on $X$ and the Krull
topology on $\Gamma_k$). Finally let (free $\Z[\Gamma_k]$-mod) denote
the full subcategory consisting of $X$ which are free as
$\Z$-modules.  To each diagonalizable group $D$ over $k$, we
associate the character group $X^*(D)$ and the cocharacter group
$X_*(D)$ as follows:
\[ X^*(D):=\Hom_{\bar k}(D,\Gm), \quad  X_*(D):=\Hom_{\bar
  k}(\Gm,D). \]
These are finitely generated $\Z$-modules equipped with a continuous
$\Gamma_k$-action.

We have the following basic results (see
\cite{borel:lag}).
\begin{thm}\label{2.1}\

\begin{enumerate}
\item The functor $X^*$ gives rise to an anti-equivalence of abelian
  categories between {\rm (Diag.~groups$/k$)} and
  {\rm ($\Z[\Gamma_k]$-mod)} which preserves the short exact
  sequences. Moreover, it induces an anti-equivalence between the
  categories {\rm (Tori$/k$)} and {\rm (free $\Z[\Gamma_k]$-mod)}. 
\item The functor $X_*$ gives rise to an equivalence of 
  categories between {\rm (Tori$/k$)} and
  {\rm (free $\Z[\Gamma_k]$-mod)}.
\item For any diagonalizable group $D$ over $k$, we have a canonical 
  isomorphism
\[ X_*(D)\simeq \Hom_{\Z}(X^*(D),\Z) \]
of $\Z[\Gamma_k]$-modules.  
\end{enumerate}
\end{thm}

Using this theorem, one obtains the following consequence immediately.

\begin{cor}\label{2.2}
Let $T$ be an algebraic torus over $k$.  
\begin{enumerate}
\item There is a natural bijection between the set of algebraic
  $k$-subtori of $T$ and the set of quotient $\Z[\Gamma_k]$-modules of
  the character group $X^*(T)$ which are free $\Z$-modules.
\item There is a natural bijection between the set of algebraic
  $k$-subtori of $T$ and the set of saturated
  $\Z[\Gamma_k]$-submodules of 
  the cocharacter group $X_*(T)$.  
\end{enumerate}
\end{cor}

A $\Z$-sublattice $L$ of a $\Z$-lattice $X$ 
is called {\it saturated} if the quotient abelian group 
$X/L$ is torsion-free. For
any $\Z$-sublattice $L$ of $X$ there is a unique maximal $\Z$-sublattice
$L'$ in $X$ of the same rank that contains $L$; this lattice is
saturated and is called {\it the saturation of $L$ in $X$}. Clearly,
$L'$ can be constructed by putting $L'=L\cdot \Q\cap X$. 

Let $\alpha:S\to T$ be a homomorphism of algebraic tori over $\bar
k$. We have the induced homomorphisms 
\[ \alpha_*:X_*(S)\to X_*(T) \quad \text{and}\quad  \alpha^*:
X^*(T)\to X^*(S) \]
given by $\alpha_*(\gamma)=\alpha\circ \gamma$ and $\alpha^*
\chi=\chi\circ \alpha$ for $\gamma\in X_*(S)$ and $\chi\in X^*(T)$:
\[ 
\begin{CD}
  \Gm @>\gamma>> S @>\alpha>> T
\end{CD}, \quad  
\begin{CD}
  S @>\alpha>> T @>\chi>> \Gm.
\end{CD} 
\]
The natural perfect 
pairing $X^*(T)\times X_*(T)\to \Z=\End(\Gm)$ is simply
the composition $\<\chi,\gamma\>_T=\chi\circ \gamma$.
Then we have the adjoint property:
\begin{equation}
  \label{eq:2.1}
  \<\chi, \alpha_* \gamma\>_T=\<\alpha^* \chi, \gamma\>_S, \quad
  \forall\, \chi\in X^*(T), \ \gamma\in X_*(S).   
\end{equation}
Indeed, we check $\<\chi, \alpha_* \gamma\>_T=\chi\circ \alpha\circ
\gamma=\<\alpha^* \chi, \gamma\>_S$.

\subsection{Algebraic $\Q$-tori}
\label{sec:2.2}
\def\pr{{\rm pr}}
Let $\Qbar$ denote the field of algebraic numbers in $\C$. All number
fields considered in this paper are those contained in $\C$.
For any number field $K$, denote by $T^K=\Res_{K/\Q} \bbG_{m,K} $ the
associated algebraic $\Q$-torus. Denote by 
$\Sigma_K:=\Hom_\Q(K,\C)=\Hom(K,\Qbar)$ the set of embeddings of $K$;
it is equipped with the $\Gamma_{\Q}$-action by $\sigma\cdot
\phi=\sigma\circ \phi$, where $\sigma\in \Gamma_{\Q}$ and $\phi\in
\Sigma_K$.
We have an isomorphism of $\Qbar$-algebras:
\[ c: K\otimes \Qbar\simeq \Qbar^{\Sigma_K}, \quad a\otimes x \mapsto
(\phi(a)x)_{\phi}. \]
The projection at the $\phi$-component via the isomorphism $c$ gives 
a $\Qbar$-algebra homomorphism 
\[ \pr_\phi:K\otimes \Qbar\simeq \Qbar^{\Sigma_K}\to \Qbar. \]  
This defines a character and we denote this again by $\phi\in
X^*(T^K)$. Clearly $\Sigma_K$ forms a $\Z$-basis for $X^*(T^K)$.
The action ${}^\sigma\! \phi$ is defined (as the Galois action on the set
of morphisms of varieties over $\Qbar$) by the commutative diagram
\[ 
\begin{CD}
K\otimes \Qbar @>\phi>> \Qbar \\
   @V 1\otimes\sigma VV  @V \sigma VV \\
K\otimes \Qbar @>{{}^\sigma\! \phi}>> \Qbar.  
\end{CD}
\] 
We see $({}^\sigma\phi)(a\otimes \sigma(x))=\sigma (\phi(a)x)=\sigma
\phi(x)\cdot \sigma(x)$; so ${}^\sigma \phi=\sigma \phi=\sigma\cdot
\phi$ (the latter is the natural action by $\Gamma_\Q$). We obtain
$X^*(T^K)=\Z[\Sigma_K]$, the free $\Z$-module generated by $\Sigma_K$,
with the $\Gamma_\Q$-action by ${}^\sigma \!\phi=\sigma \phi$.  

Let $\Sigma_K^\vee:=\{\phi^\vee\ ;\, \phi\in \Sigma_K\}$ be the basis
of $\Hom(\Z[\Sigma_K],\Z)$ dual to $\Sigma_K$. 
As the pairing $X^*(T)\times X_*(T)\to \Z$ is
$\Gamma_\Q$-equivariant we have
\[ \<\phi', {}^\sigma \phi^\vee\>=\<{}^{\sigma^{-1}}\!\! \phi',
  \phi^\vee\>=
  \begin{cases}
    1, & \text{if ${}^{\sigma^{-1}}\!\! \phi'=\phi$ (or $\phi'={}^\sigma
    \phi$)},  \\
    0, & \text{otherwise.}
  \end{cases} \]
It follows that ${}^\sigma\! \phi^\vee=(\sigma \phi)^\vee$. So we obtain
\begin{equation}
  \label{eq:2.2}
  X_*(T^K)=\Z[\Sigma_K^\vee], \quad \text{with}\ {}^\sigma\!
\phi^\vee=(\sigma\phi)^\vee, \quad \forall\,\sigma\in \Gamma_\Q.
\end{equation}

Let $K\subset k$ be two number fields. We have the inclusion
$T^K\subset T^k$. The induced homomorphism $X^*(T^k)\to X^*(T^K)$ is
given by the restriction map $\Sigma_k\to \Sigma_K$, $\wt \phi\mapsto
\wt \phi|_K$. 
For any subset $\Phi\subset \Sigma_K$, we denote by $\wt \Phi\subset
\Sigma_k$ the preimage of this restriction map. One easily checks that
the cocharacter groups $X_*(T^K)\subset X_*(T^k)$ is given by the
relation
\begin{equation}
  \label{eq:2.3}
  \phi^\vee=\sum_{\wt \phi|_K=\phi} \wt \phi^\vee. 
\end{equation}

%\[ \sum_{\wt \phi} n_{\wt \phi} {\wt \phi} \mapsto \sum_{\wt \phi}
% n_{\wt \phi} {\wt \phi}|_K=\sum_{\phi} (\sum_{\wt \phi; \wt
%   \phi|_K=\phi} n_{\wt \phi} \phi. \] 

\section{The Main Theorem of Complex Multiplication}
\label{sec:03}

In this section we describe the main theorem of complex multiplication
due to Shimura and Taniyama~\cite{shimura-taniyama}. Our main
reference is Serre and Tate~\cite{serre-tate}. 

\subsection{The reflex type norms}
\label{sec:3.1}
Let $K$ be a CM field and $\Phi\subset \Sigma_K$ be a CM type,
i.e. the complement of $\Phi$ equals its complex conjugate. 
Let $(A,i)$ be an abelian variety over a number field $k$ of CM-type
$(K,\Phi)$. That is,  
\[ i:K\to \End_k^0(A):=\End_k(A)\otimes_\Z \Q \]
is a ring monomorphism, $2 \dim A=[K:\Q]$, and the character of the
representation of $K$ on the Lie algebra $\Lie(A_\C)$ is given by
$\sum_{\varphi\in \Phi} \varphi$ (In case the action $i$ of
$K$ is defined over $\bar k$ instead of $k$, i.e. $i:K\to
\End^0(A\otimes \bar k)$, the pair 
$(A,i)$ is said to be potentially of CM type
$(K,\Phi)$). Note that  
any complex abelian variety of CM type 
is defined over a number field \cite{shimura-taniyama}.
Replacing $k$ by a finite extension of $k$, we assume that 
$k$ is Galois over $\Q$ and contains
$K$. We usually identify $K$
with its image under $i$, that is, we view $i$ as the inclusion. 
The endomorphism algebra $\End^0_k(A)$ acts on $V:=H_1(A(\C),\Q)$ 
so that $V$ becomes a free $K$-module of rank one.  
We have the inclusion
$K\subset \End(V)$ and may regard the algebraic torus $T^K\subset
\GL(V)$ as an algebraic subgroup over $\Q$.

Set 
\[ \Phi_k:=\{\wt \phi\in \Sigma_k\ ;\ \wt \phi|_K\in \Phi \}, \] 
that is, $\Phi_k=\wt \Phi$ with respect 
to the restriction map $\Sigma_k\to \Sigma_K$.
Put $H:=\Gal(k/K)$. Let 
\[ H_E:=\{\sigma\in \Gal(k/\Q)\ ;\ \sigma \Phi=\Phi \}, \] 
and let $E\subset \Qbar$ be the fixed field of $H_E$, 
the {\it reflex field of the pair $(K,\Phi)$}. 
Notice that $\Sigma_k=\Gal(k/\Q)$ is a group
and that the inverse set $\Phi_k^{-1}$ is right stable under the
subgroup $H_E$. Therefore, the set $\Phi_k^{-1}$ descends to a unique
CM type 
$\Phi_E\subset \Sigma_E$. The norm map 
$ N_{\Phi_k^{-1}}: k^\times \to k^\times$ given by
\begin{equation}
  \label{eq:3.1}
  N_{\Phi_k^{-1}}(x):=\prod_{\sigma\in \Phi_k^{-1}} \sigma(x)
\end{equation}
defines a homomorphism 
$N_{\Phi_k^{-1}}: T^k \to T^k$ of algebraic tori over $\Q$. It factors
  through the subtorus $\psi: T^k \to T^K$ and we have the following
  commutative diagram
\begin{equation}
  \label{eq:3.2}
\begin{CD}
   T^k @>N_{\Phi_k^{-1}}>> T^k \\
   @VN_{k/E} VV  \cup   \\
   T^E @>{N_{\Phi_E}}>> T^K, \\      
\end{CD}  
\end{equation}
where $N_{k/E}$ is the usual norm map and $N_{\Phi_E}$ is the reflex
norm map defined in
the same manner as $N_{\Phi_k^{-1}}$ ($\psi=N_{\Phi_E}\circ N_{k/E}$). 
Remark that $\psi$ is the same homomorphism constructed by the
determinant map using the $(K,k)$-bimodule structure 
of the Lie algebra $\Lie(A/k)$ of $A$ in
Serre and Tate \cite[Section 7]{serre-tate}.
 
Let %There are homomorphisms
\begin{equation}
  \label{eq:3.3}
  \begin{array}[l]{l}
  \psi_0: k^\times \to K^\times, \\
  \psi_\ell: k_\ell^\times \to K^\times_\ell, \\
  \psi_\infty: k_\infty^\times \to K^\times_\infty, 
  \end{array}
\end{equation}
be the homomorphisms induced from the morphism $\psi:T^k\to T^K$ 
by evaluating at $\Q$, $\Q_\ell$ and $\R$, respectively, 
where $k_\ell=\Q_\ell\otimes k$, $K_\ell=\Q_\ell\otimes K$,
$k_\infty=\R\otimes k$ and $K_\infty=\R\otimes K$.

\subsection{The explicit reciprocity law}
\label{sec:3.2}

Let $V^k$ denote the set of all places of $k$, $V^k_\infty$
(resp. $V^k_f$) the set of
all archimedean (resp. finite) 
places of $k$, and $V^k_\ell$ the set of the finite places
lying above $\ell$. Let $S_A$ be the finite set of finite places where
$A$ has bad reduction.  

If $v\in V^k_f-S_A$, let $k(v)$, $A(v)$, $\pi_{A(v)}$ denote
respectively the residue field at $v$, the reduction of $A$ at $v$,
the Frobenius endomorphism of $A(v)$ relative to $k(v)$. The reduction
map $\End(A)\to \End(A(v))$ defines an injection
\[ i_v: K \to \End^0(A) \to \End^0(A(v)). \]
Since $\pi_{A(v)}$ commutes with every $k(v)$-endomorphism of $A(v)$,
it lies in the commutator of the image ${\rm Im} (i_v)$ of $i_v$,
which is 
again ${\rm Im}(i_v)$. Thus there is an unique element $\pi_v\in K$
such 
that $i_v(\pi_v)=\pi_{A(v)}$; we call $\pi_v$ the {\it Frobenius
element} attached to $v$. 

Let $I_k$ denote the id\`ele group of $k$. For each finite set
$S\subset V^k$, let $I_k^S\subset I_k$ denote the group of id\`eles
$(a_v)$ such that $a_v=1$ for all $v\in S$.   

The next two theorems, due to Serre and Tate~\cite[Theorems 10 and 11,
Section 7]{serre-tate}, are reformulation of
results of Shimura-Taniyama \cite{shimura-taniyama} and Weil
\cite{weil:1955}.  

\begin{thm}\label{3.1}
There exists a unique homomorphism 
\begin{equation}
  \label{eq:3.4}
  \varepsilon: I_k \to K^\times 
\end{equation}
satisfying the following three conditions:
\begin{itemize}
\item [(a)] The restriction of $\varepsilon$ to $k^\times$ is the map
  $\psi_0: k^\times \to K^\times$ defined in (\ref{eq:3.3}).
\item [(b)] The homomorphism $\varepsilon$ is continuous, in the sense
  that its kernel is open in $I_k$. 
\item [(c)] There is a finite set $S\subset V^k$ containing
  $V^k_\infty$ and $S_A$ such that 
  \begin{equation}
    \label{eq:3.5}
    \varepsilon(a)=\prod_{v\not\in S} \pi_v^{v(a_v)}, \quad \forall\,
    a\in I_k^S.
  \end{equation}
\end{itemize}
\end{thm}

The last condition means that, 
for $v\not\in S$ one has $\varepsilon(\varpi_v)=\pi_v$, where
$\varpi_v$ is any uniformizer of the completion $k_v$ at $v$, . 

Let 
\[ \rho_\ell:\Gamma_k \to \GL(V_\ell) \]
be the $\ell$-adic Galois representation associated to the Tate module
$V_\ell(A)$. As the image of $\rho_\ell$ is contained in
$K_\ell^\times$, the map $\rho_\ell$ factors 
through $\Gal(k^{\rm ab}/k)$, where
$k^{\rm ab}$ is the maximal abelian extension of $k$. Class field
theory allows us to interpret $\rho_\ell$ as a homomorphism
\[ \rho_\ell: I_k \to K_\ell^\times \]
which is trivial on $k^\times$. If $v\not\in (S_A\cup V^k_\ell)$, then
$\rho_v$ is {\it unramified at $v$} (i.e. $\rho_\ell$ is trivial on
$\calO_v^\times$) and $\rho_\ell(\varpi_v)=\pi_v$. (The normalization
of the Artin map $I_k\to \Gal(k^{\rm ab}/k)$ used in
Serre-Tate~\cite{serre-tate} 
sends each uniformizer $\varpi_v$ to the {\it arithmetic} Frobenius
automorphism ${\rm Fr}_v$.)   

\begin{thm}\label{3.2}
  For each prime number $\ell$, we have 
  \begin{equation}
    \label{eq:3.6}
    \rho_\ell(a)=\varepsilon(a) \psi_\ell(a_\ell^{-1}),
    \quad \forall\, a\in I_k,     
  \end{equation}
  where $a_\ell$ denotes the component of $a$ in $k_\ell^\times$, and
  $\psi_\ell:k_\ell^\times \to K_\ell^\times$ is the map defined
  in (\ref{eq:3.3}). 
\end{thm}

\section{Proof of MTC for CM abelian varieties}
\label{sec:04}

We preserve the notations and hypotheses in the previous section.
Let $T_0$ be the image of the homomorphism $\psi:T^k\to T^K$. The
Mumford-Tate conjecture for the abelian variety $(A,i)$ of CM type,
i.e. $G_{A,\ell}^0=MT(A)\otimes \Q_\ell$, 
will follow from the following two lemmas.

\begin{lemma}\label{4.1}
  We have $G^0_{A,\ell}=T_0\otimes \Q_\ell$.
\end{lemma}
\begin{proof}
  By Theorem~{3.2}, the map $\rho_\ell$ agrees with $\psi_\ell^{-1}$ on an
  open subgroup $U_\ell$ of $k_\ell^\times$. Since $U_\ell$ is Zariski
  dense in $T^k$, the Zariski closure of the image $\rho_\ell(I_k)$
  contains the image of $\psi_\ell$ (as an algebraic $\Q_\ell$-torus), 
  which is $T_0\otimes \Q_\ell$. 
  This shows the inclusion $T_0 \otimes \Q_\ell\subset G^0_{A,\ell}$. 

  On
  the other hand, the map $\rho_\ell:I_k\to T^K(\Q_\ell)/T_0(\Q_\ell)$
  factors through the quotient:
  \begin{equation}
    \label{eq:4.1}
   \rho_\ell: I_k/k^\times U \to
  T^K(\Q_\ell)/T_0(\Q_\ell), 
  \end{equation}
where $U$ is the kernel of $\varepsilon$. Notice $U\supset
(k_\infty^\times)^0$. By the finiteness of 
class numbers, the group $I_k/k^\times U$ is finite.
Therefore, $\rho_\ell$ has finite image in
$T^K(\Q_\ell)/T_0(\Q_\ell)$. This shows the inclusion 
$G^0_{A,\ell}\subset T_0\otimes \Q_\ell$. 

We then conclude the equality $G^0_{A,\ell}=T_0\otimes
\Q_\ell$. \qed  
\end{proof}

\begin{lemma}\label{4.2}
  We have $MT(A)=T_0$.
\end{lemma}
\begin{proof}
The Hodge cocharacter $\mu:\C^\times \to (K\otimes \C)^\times$ has the
property 
\[ \<\phi, \mu\>=
\begin{cases}
  1 & \text{if $\phi\in \Phi$;} \\
  0 & \text{otherwise.}
\end{cases}\]
Therefore $\mu=\sum_{\phi\in \Phi} \phi^\vee$. 
As $MT(A)$ is the smallest $\Q$-torus of $T^K$ containing the 
image of $\mu$, its cocharacter group is equal to 
the saturation of the sublattice
\[ \Z\,[\, {}^\sigma\! \mu\ ;\ \sigma\in \Gal(k/\Q)\, ]\subset
X_*(T^K). \] 

Now we want to determine the image of the map
$X_*(T^k)\to X_*(T^K)\subset X_*(T^k)$ induced by $\psi$. 
Note that the map induced by $\psi$ is that induced by the reflex norm
map $N_{\Phi_k^{-1}}$. We claim 
\begin{equation}
  \label{eq:4.2}
  \psi_*(\tau^\vee)={}^\tau\! \mu \quad 
\text{in $X_*(T^K)=\Z[\Sigma_K^\vee]\subset \Z[\Sigma_k^\vee]$}
\end{equation}
for every $\tau\in \Sigma_k$.

For $\sigma\in \Gal(k/\Q)$, the isomorphism $\sigma:k\to k$ induces an
isomorphism $\sigma:T^k\to T^k$ of algebraic $\Q$-tori. The pullback
map $\sigma^*: X^*(T^k)\to X^*(T^k)$ is given by $\sigma^* \tau=\tau
\sigma$. 
For the push-forward map $\sigma_*$ we have 
$\sigma_* \tau^\vee=(\tau\circ \sigma^{-1})^\vee$ for $\tau\in
\Sigma_k$; this follows from 
\[ \<\tau', \sigma_* \tau^\vee\>=\<\tau' \sigma, \tau^\vee\>=
\begin{cases}
  1 & \text{if $\tau' \sigma=\tau$ (or $\tau'=\tau \sigma^{-1}$);}\\
  0 & \text{otherwise.} 
\end{cases} \]
%we get $\sigma_* \tau^\vee=(\tau\circ \sigma^{-1})^\vee$ for $\tau\in
%\Sigma_k$. 
We compute
\[ (N_{\Phi_{k}^{-1}})_*(\tau^\vee)=\sum_{\sigma\in \Phi_k^{-1}}
  (\tau\circ \sigma^{-1})^\vee=\sum_{\sigma\in \Phi_k} (\tau\circ
  \sigma)^\vee=\sum_{\phi\in \Phi} (\tau\circ \phi)^\vee. \]
On the other hand we also have 
\[ {}^\tau \mu=\sum_{\phi\in \Phi} {}^\tau \phi^\vee=\sum_{\phi\in
  \Phi} (\tau\circ \phi)^\vee. \]
This proves our claim. 
%\begin{equation}
%  \label{eq:4.2}
%  \psi_*(\tau^\vee)={}^\tau\! \mu.
%\end{equation}

It follows from our claim that the torus $T_0$ corresponds 
to the saturation of the $\Z$-sublattice 
\[ \Z\,[\,{}^\tau\! \mu\ ;\ \tau\in \Gal(k/\Q)\,]\subset X_*(T^K). \]
By Corollary~\ref{2.2}, one has $MT(A)=T_0$.  \qed
\end{proof}

\section{The MTC for CM motives}
\label{sec:05}

In this section we extend Pohlmann's theorem to CM motives.
It is more convenient to extend the CM field $K$ to a CM algebra, that
is, a product of CM fields. 
We consider abelian varieties of CM type $(K,\Phi)$, where $K$ is a CM
algebra and $\Phi$ is a CM type. As the base field $k$ does not play a
role in the MTC, we may assume that 
it is Galois over $\Q$ containing $k$ and that the monodromy group 
$\rho_\ell(\Gamma_k)$ is contained in $G_{A,\ell}^0(\Q_\ell)$. 
It is easy to show the MTC for
abelian varieties of CM type can be reduced to the case where $K$ is a
CM field and thus holds.     
 
By a CM motive we mean a CM motives in Andr\'e's category of motives 
with respect to motivated cycles. We refer to \cite{andre:motif} or 
\cite{deligne:period} for the formal definition. 
A result states that every CM motive is a
summand of a tensor product of copies of the motive of a 
abelian variety of CM type 
(see \cite[Proposition 1.1]{ogus:chowla-selberg}).  

Suppose $M$ is a CM motive over a number field $k$. Then replacing $k$
by a suitable finite extension if necessary, there is an
abelian variety of CM type over $k$ so that $M$ is a summand of a
tensor product of the motive $h(A)$ of $A$. As before, let 
$V=H_1(A_\C,\Q)$ and we identify
$V_\ell:=V\otimes \Q_\ell$ with $V_\ell(A)$. Write 
\[ V(m,n,r)=V^{\otimes m} \otimes_\Q \check V^{\otimes
  n} \otimes_\Q \Q(r) \text{\ and\ } 
V_\ell(m,n,r)=V_\ell^{\otimes m} 
\otimes_{\Q_\ell} \check V_\ell^{\otimes n} \otimes_{\Q_\ell}
\Q_\ell(r), \]
respectively, where $m$ and $n$ are nonnegative integers and $r\in
\Z$. We have a canonical isomorphism 
$V(m,n,r)\otimes \Q_\ell\simeq V_\ell(m,n,r)$.

As a remark in the Introduction, the Tannakian
category $\calT(V)$ generated by the Hodge structure $V\oplus \Q(1)$
is equivalent to the Tannakian category $\calT(V_\ell)$ generated by
the $\Q_\ell$-representation $V_\ell\oplus \Q_\ell(1)$ of $\Gamma_k$.     
As a result, for any Hodge substructure $W\subset V(m,n,r)$, the
$\Q_\ell$ subspace $W\otimes \Q_\ell\subset V(m,n,r)\simeq
V_\ell(m,n,r$ is stable under the $\Gamma_k$-action. Conversely 
any subrepresentation $W_\ell$ of $V_\ell(m,n,r)$ is of the form
$W\otimes \Q_\ell$ for a unique Hodge substructure $W\subset
V(m,n,r)$.      

To prove the MTC for the CM motive $M$, one may show that the Tannakian
categories $\calT(M_B)$ and $\calT(M_\ell)$ are equivalent, where
$M_B$ and $M_\ell$ are the Betti and $\ell$-adic realizations of $M$,
respectively. Suppose that $W$ is an object in $\calT(M_B)$. Then
$W\otimes \Q_\ell$ gives an object in $\calT(M_\ell)$ since $W\subset
M_B(m,n,r)$ for some integers $m,n,r$ and hence $W\subset V(m',n', r')$
for some other integers $m',n',r'$. Conversely, any object
$W_\ell\subset M_\ell(m,n,r)\subset V_\ell(m',n',r')$ is of the form
$W\otimes \Q_\ell$ for a unique object $W\subset V(m',n',r')$. It
follows from the uniqueness that $W\subset M_B(m,n,r)$. This proves
the equivalence of $\calT(M_B)$ and $\calT(M_\ell)$ and hence the MTC
for $M$. 

\section*{Acknowledgments}
%\begin{thanks}
  The present article is prepared and revised while the author's stays 
  in the RIMS, Kyoto University and in the Max-Planck-Institut f\"ur
  Mathematik in Bonn. 
  He thanks Akio Tamagawa for helpful discussions and 
  both institutions for kind hospitality 
  and excellent working conditions. Special thanks go to the referee
  for careful reading and helpful suggestions which improve 
  results of the paper.   
%  explanation of H.~Pohlmann's contribution. 
%  This corrects a serious mistake in the previous version.    
  The author is partially supported by the grants MoST
  100-2628-M-001-006-MY4 and 103-2918-I-001-009.
%\end{thanks}

\end{document}